\documentclass[12pt]{article}
\usepackage[utf8]{inputenc}
\usepackage[T1]{fontenc}
\usepackage{amsmath,amssymb,amsthm}
\usepackage{geometry}
\usepackage{hyperref}
\usepackage{indentfirst}
\usepackage{booktabs}
\usepackage{xcolor}
\usepackage{fvextra}
\DefineVerbatimEnvironment{code}{Verbatim}{
  breaklines=true,
  breakanywhere=false,
  fontsize=\footnotesize,
  xleftmargin=1em,
}
\DefineVerbatimEnvironment{term}{Verbatim}{
  breaklines=true,
  breakanywhere=true,
  fontsize=\footnotesize,
  xleftmargin=1em,
}

\newtheorem{theorem}{Theorem}
\newtheorem{lemma}{Lemma}
\newtheorem{proposition}{Proposition}
\newtheorem{corollary}{Corollary}
\theoremstyle{remark}
\newtheorem{remark}{Remark}
\theoremstyle{definition}
\newtheorem{definition}{Definition}

\title{Non-Existence of Quintic Factorization for the Second Cuboid Polynomial $Q_{p,q}(t)$}

\author{Valery Asiryan\\[3pt]
\small \texttt{asiryanvalery@gmail.com}}
\date{\small January 3, 2026}

\begin{document}

\maketitle

\begin{abstract}
We consider the even monic degree-$10$ \emph{second cuboid polynomial} $Q_{p,q}(t)\in\mathbb{Z}[t]$ depending on coprime integers $p\neq q>0$.
We exclude the existence of a splitting of type $5+5$ over $\mathbb{Q}$, i.e., a factorization of $Q_{p,q}(t)$ into two irreducible quintic polynomials.
Since $Q_{p,q}(t)$ is even and satisfies $Q_{p,q}(0)\neq 0$, any such $5+5$ splitting is necessarily \emph{symmetric}, meaning that it can be written in the normal form
\[
Q_{p,q}(t)=R_{p,q}(t)\cdot \bigl(-R_{p,q}(-t)\bigr),\qquad \deg R_{p,q}=5.
\]
After a weighted normalization reducing to a one-parameter polynomial $Q_r(u)$ with $r=p/q\in\mathbb{Q}_{>0}$, coefficient comparison and elimination via resultants show that a $5+5$ splitting forces the existence of a rational point on an explicitly defined plane curve $F(r,a)=0$.
Passing to the quotient parameters $a=r y$ and $s=r^2$ yields an affine curve $f(s,y)=0$ such that, for each fixed $s>0$, the polynomial $f(s,\cdot)$ is of degree $16$.
We compute and factor the discriminant $\mathrm{Disc}_y(f)$ and then use Sturm root counts to certify that $f(s,\cdot)$ has no real roots for every rational $s>0$ with $s\neq 1$.
Hence $f(s,y)=0$ admits no rational solutions with $s>0$, $s\neq 1$, and consequently no quintic $5+5$ factorization occurs for $Q_{p,q}(t)$ when $p\neq q$.


\end{abstract}


\section{Introduction}\label{sec:intro}
The perfect cuboid problem asks for a rectangular box with integer edges whose face diagonals and space diagonal all have integer lengths.
Several approaches encode this problem into Diophantine conditions on parameter-dependent polynomials.
In the framework introduced by R.\,A.\,Sharipov \cite{Sharipov2011Cuboids,Sharipov2011Note}, one is led to a small collection of explicit even polynomials whose irreducibility over $\mathbb{Z}$ is conjectured for coprime parameters.
The present note addresses one specific splitting pattern (degree $5+5$) for the \emph{second cuboid polynomial} $Q_{p,q}(t)$.

\medskip
\noindent\textbf{Goal.}
We prove that for coprime integers $p\neq q>0$ the polynomial $Q_{p,q}(t)$ admits no \emph{quintic} factorization of type $5+5$ over $\mathbb{Q}$ (hence none over $\mathbb{Z}$), i.e., it cannot be written as a product of two irreducible quintic polynomials.
By Lemma~\ref{lem:sym_normal}, any such $5+5$ splitting would have to be symmetric, so it suffices to exclude factorizations of the form
\[
Q_{p,q}(t)=R_{p,q}(t)\cdot \bigl(-R_{p,q}(-t)\bigr),\qquad \deg R_{p,q}(t)=5.
\]
The proof is explicit and certificate-style: the final obstruction is an interval-by-interval Sturm verification based on the discriminant stratification of a one-parameter family of degree-$16$ polynomials.

\medskip
\noindent\textbf{Outline.}
Section~\ref{sec:poly} records $Q_{p,q}(t)$ and a weighted normalization to $Q_r(u)$.
Section~\ref{sec:sym55} shows that any quintic $5+5$ splitting is necessarily symmetric and reduces it to a resultant condition $F(r,a)=0$.
Section~\ref{sec:quotient} introduces the quotient parameters $(s,y)$ and the curve $f(s,y)=0$.
Section~\ref{sec:sturm} proves that $f(s,\cdot)$ has no real roots for all rational $s>0$, $s\neq 1$ by Sturm root counts, yielding the main theorem.
Appendix~\ref{app:scripts} contains the SageMath/SymPy scripts used to reproduce the computations together with their transcripts.


\section{The second cuboid polynomial and normalization}\label{sec:poly}

\subsection{Definition of \texorpdfstring{$Q_{p,q}(t)$}{Qpq(t)}}
Let $p,q\in\mathbb{Z}_{>0}$ be coprime and $p\neq q$.
The second cuboid polynomial is the even monic degree-$10$ polynomial
\begin{align}
Q_{p,q}(t)
={}&t^{10} + (2 q^{2} + p^{2})(3 q^{2} - 2 p^{2})\, t^{8} \nonumber\\
&+ (q^{8} + 10 p^{2} q^{6} + 4 p^{4} q^{4} - 14 p^{6} q^{2} + p^{8})\, t^{6} \nonumber\\
&- p^{2} q^{2}\,(q^{8} - 14 p^{2} q^{6} + 4 p^{4} q^{4} + 10 p^{6} q^{2} + p^{8})\, t^{4} \nonumber\\
&- p^{6} q^{6}\,(q^{2} + 2 p^{2})(-2 q^{2} + 3 p^{2})\, t^{2}
- p^{10} q^{10}\in\mathbb{Z}[t].\label{eq:Qpq}
\end{align}

\subsection{Weighted normalization}
The polynomial \eqref{eq:Qpq} is weighted-homogeneous of total weight $20$ with weights
\[
\deg(p)=\deg(q)=1,\qquad \deg(t)=2.
\]
Consequently one can normalize $q=1$ by a scaling of $t$.

\begin{lemma}[Normalization to a one-parameter family]\label{lem:normalize}
Let $q\neq 0$ and set
\[
r:=\frac{p}{q}\in\mathbb{Q},\qquad u:=\frac{t}{q^{2}}\in\mathbb{Q}.
\]
Then
\begin{equation}\label{eq:normalize}
Q_{p,q}(t)=q^{20}\,Q_r(u),
\end{equation}
where
\begin{equation}\label{eq:Qr}
\begin{aligned}
Q_r(u)
={}& u^{10} + (2+r^2)(3-2r^2)\,u^{8} + \bigl(1+10r^2+4r^4-14r^6+r^8\bigr)\,u^{6}\\
& - r^2\bigl(1-14r^2+4r^4+10r^6+r^8\bigr)\,u^{4}
 - r^6(1+2r^2)(-2+3r^2)\,u^{2}
 - r^{10}\in\mathbb{Q}[u].
\end{aligned}
\end{equation}
\end{lemma}

\begin{proof}
Substitute $p=rq$ and $t=q^2 u$ into \eqref{eq:Qpq} and factor out $q^{20}$, using the weight check above.
\end{proof}

\begin{corollary}[Integral vs.\ rational reduction of the symmetric normal form]\label{cor:int_to_rat}
Let $p,q\in\mathbb{Z}_{>0}$ with $q\neq 0$ and set $r=p/q$ and $u=t/q^2$ as in Lemma~\ref{lem:normalize}.
If there exists a quintic $R_{p,q}(t)\in\mathbb{Z}[t]$ such that
\[
Q_{p,q}(t)=R_{p,q}(t)\cdot\bigl(-R_{p,q}(-t)\bigr),
\]
then $Q_r(u)$ admits a symmetric $5+5$ factorization over $\mathbb{Q}$.
More precisely, defining
\[
\widetilde{R}(u):=q^{-10}\,R_{p,q}(q^{2}u)\in\mathbb{Q}[u],
\]
we have
\[
Q_r(u)=\widetilde{R}(u)\cdot\bigl(-\widetilde{R}(-u)\bigr).
\]
\end{corollary}

\begin{proof}
Substitute $t=q^2u$ into the assumed factorization and divide by $q^{20}$.
Since $R_{p,q}(q^2u)=q^{10}\widetilde{R}(u)$, we obtain
\[
q^{20}Q_r(u)=q^{10}\widetilde{R}(u)\cdot\bigl(-q^{10}\widetilde{R}(-u)\bigr),
\]
hence the claimed identity in $\mathbb{Q}[u]$.
\end{proof}

\begin{remark}\label{rem:domain}
Throughout the paper we focus on $r\in\mathbb{Q}_{>0}$ (since $p,q>0$) and $r\neq 1$ (since $p\neq q$).
\end{remark}


\section{Quintic \texorpdfstring{$5+5$}{5+5} factorization and the elimination curve}\label{sec:sym55}

\subsection{Reduction to the symmetric normal form}

\begin{definition}[Quintic $5+5$ factorization]\label{def:55}
Let $K$ be a field of characteristic $0$ and let $Q(u)\in K[u]$ be an even monic polynomial of degree $10$ with $Q(0)\neq 0$.
We say that $Q$ admits a \emph{quintic $5+5$ factorization over $K$} (or a \emph{$5+5$ splitting}) if there exist monic \emph{irreducible} quintics $A(u),B(u)\in K[u]$ such that
\[
Q(u)=A(u)B(u),\qquad \deg A=\deg B=5.
\]
\end{definition}

\begin{lemma}[Symmetric normal form]\label{lem:sym_normal}
In the setting of Definition~\ref{def:55}, if $Q$ admits a quintic $5+5$ factorization over $K$, then there exists a monic quintic $R(u)\in K[u]$ such that
\[
Q(u)=R(u)\cdot \bigl(-R(-u)\bigr).
\]
\end{lemma}

\begin{proof}
Let $Q(u)=A(u)B(u)$ with $A,B$ monic irreducible quintics.
Since $Q$ is even, we have $Q(u)=Q(-u)=A(-u)B(-u)$.
Because $Q(0)\neq 0$, we have $A(0)B(0)=Q(0)\neq 0$, hence $A(0)\neq 0$ and $B(0)\neq 0$.

The polynomial $A(-u)$ is irreducible of degree $5$ and divides $Q$.
By unique factorization in $K[u]$, it must be an associate of $A$ or of $B$.

It cannot be an associate of $A$: indeed, if $A(-u)=\pm A(u)$, then $A$ is even or odd.
An odd polynomial satisfies $A(0)=0$, contradiction.
An even polynomial cannot have odd degree $5$ unless it is identically zero, also impossible.
Therefore $A(-u)$ is associate to $B$.

Because $A$ and $B$ are monic while $A(-u)$ has leading coefficient $-1$, we obtain $B(u)=-A(-u)$, and the claim follows with $R=A$.
\end{proof}

\subsection{The symmetric ansatz}
\begin{definition}[Symmetric $5+5$ factorization]\label{def:sym55}
Let $K$ be a field of characteristic $0$ and let $Q(u)\in K[u]$ be an even polynomial of degree $10$ with $Q(0)\neq 0$.
We say that $Q$ admits a \emph{symmetric $5+5$ factorization over $K$} if there exists a quintic polynomial
\[
R(u)=u^5+a u^4+b u^3+c u^2+d u+e\in K[u]
\]
such that
\begin{equation}\label{eq:sym55}
Q(u)=R(u)\cdot\bigl(-R(-u)\bigr).
\end{equation}
By Lemma~\ref{lem:sym_normal}, every quintic $5+5$ factorization in the sense of Definition~\ref{def:55} admits such a symmetric normal form.
\end{definition}

\begin{proposition}[Coefficient comparison]\label{prop:coeff}
Fix $r\in\mathbb{Q}\setminus\{0\}$.
If $Q_r(u)$ admits a quintic $5+5$ factorization over $\mathbb{Q}$ in the sense of Definition~\ref{def:55}, then there exist $(a,d)\in\mathbb{Q}^2$ such that the two equations
\begin{equation}\label{eq:E2E3}
E_2(r,a,d)=0,\qquad E_3(r,a,d)=0
\end{equation}
hold, where $E_2,E_3\in\mathbb{Z}[r,a,d]$ are the numerators of the rational expressions obtained from the coefficient comparison of
$Q_r(u)-R(u)\bigl(-R(-u)\bigr)$ after solving the linear constraints for $b,c,e$.
\end{proposition}

\begin{proof}
Assume that $Q_r(u)$ admits a quintic $5+5$ factorization over $\mathbb{Q}$ in the sense of Definition~\ref{def:55}. By Lemma~\ref{lem:sym_normal} we may write it in the symmetric normal form \eqref{eq:sym55} for some monic quintic $R(u)$, and we expand $R(u)\bigl(-R(-u)\bigr)$ and compare coefficients with \eqref{eq:Qr}.

The constant term gives $e^2=r^{10}$, so (over $\mathbb{Q}$) $e=\pm r^5$.
If $e=-r^5$, then replacing $R(u)$ by the monic quintic
\[
\widetilde R(u):=-R(-u)
\]
preserves the product \(\,R(u)\bigl(-R(-u)\bigr)\) while sending the constant term \(e\) to \(-e=r^5\).
Thus we may assume \emph{without loss of generality and without breaking monicity} that \(e=r^5\).

The $u^8$ coefficient gives a linear equation determining $b$ in terms of $r$ and $a$; the $u^2$ coefficient gives a linear equation determining $c$ in terms of $r$ and $d$.
Substituting these into the remaining two nontrivial coefficient equations yields \eqref{eq:E2E3}.
This derivation is carried out symbolically in Script~01 (Appendix~\ref{app:scripts}).
\end{proof}

\begin{proposition}[Resultant condition $F(r,a)=0$]\label{prop:F}
Let $r\in\mathbb{Q}\setminus\{0\}$.
If $Q_r(u)$ admits a quintic $5+5$ factorization over $\mathbb{Q}$ in the sense of Definition~\ref{def:55}, then there exists $a\in\mathbb{Q}$ such that
\begin{equation}\label{eq:F}
F(r,a)=0,
\end{equation}
where $F\in\mathbb{Z}[r,a]$ is the (primitive) elimination resultant
\[
F(r,a)=\mathrm{prim}\Bigl(\frac{1}{r^{20}}\mathrm{Res}_d\bigl(E_2(r,a,d),E_3(r,a,d)\bigr)\Bigr).
\]
Moreover, $F$ satisfies
\[
\deg_r F=24,\qquad \deg_a F=16,\qquad F(-r,-a)=F(r,a),
\]
and the geometric genus of the nonsingular projective model of the curve $F(r,a)=0$ equals $21$.
\end{proposition}

\begin{proof}[Computational proof]
The elimination and degree/symmetry checks, as well as the genus computation via normalization, are performed in Script~01 (Appendix~\ref{app:scripts}).
\end{proof}

\begin{remark}\label{rem:faltings}
Since the curve $F(r,a)=0$ has geometric genus $21$, Faltings' theorem implies that it has only finitely many rational points \cite{Faltings}.
However, finiteness alone does not exclude the existence of exceptional parameter values $r\in\mathbb{Q}_{>0}$ producing a quintic $5+5$ factorization.
The remainder of the paper provides an explicit exclusion for all $r\in\mathbb{Q}_{>0}$ with $r\neq 1$.
\end{remark}


\section{Quotient substitution \texorpdfstring{$(r,a)\mapsto (s,y)$}{(r,a)->(s,y)}}\label{sec:quotient}

The polynomial $F(r,a)$ is invariant under $(r,a)\mapsto (-r,-a)$.
It is therefore natural to introduce quotient parameters.

\begin{definition}\label{def:sy}
Define new variables $s,y$ by
\begin{equation}\label{eq:substy}
a=r\,y,\qquad s=r^2.
\end{equation}
Let $f(s,y)\in\mathbb{Z}[s,y]$ be defined by rewriting $F(r,r y)$ using the rule $r^{2k}\mapsto s^k$.
\end{definition}

\begin{proposition}\label{prop:f}
The polynomial $f(s,y)$ has bidegree
\[
\deg_s f=12,\qquad \deg_y f=16.
\]
Moreover, if $r\in\mathbb{Q}\setminus\{0\}$ and $F(r,a)=0$ for some $a\in\mathbb{Q}$, then with $s=r^2\in\mathbb{Q}_{>0}$ and $y=a/r\in\mathbb{Q}$ we have
\[
f(s,y)=0.
\]
\end{proposition}

\begin{proof}
The construction is purely algebraic and is executed in Script~02 (Appendix~\ref{app:scripts}), which prints the degrees.
\end{proof}


\section{Discriminant stratification and Sturm certificates}\label{sec:sturm}

Fix $s\in\mathbb{R}$ and view $f(s,y)$ as a univariate polynomial in $y$.
Let
\[
\Delta(s):=\mathrm{Disc}_y\bigl(f(s,y)\bigr)\in\mathbb{Z}[s]
\]
be its discriminant.
Sturm's theorem provides an exact algorithm to count real roots of a univariate polynomial on an interval \cite{Sturm}.

\begin{proposition}[Discriminant factorization]\label{prop:disc}
The discriminant $\Delta(s)$ factors over $\mathbb{Z}$ as
\begin{equation}\label{eq:disc}
\Delta(s)=C\cdot s^{156}(s-1)^{54}(s+1)^{22}P_6(s)^{4}P_{28}(s)^{2},
\end{equation}
where $C\in\mathbb{Z}\setminus\{0\}$ is a constant,
\begin{equation}\label{eq:P6}
P_6(s)=s^{6}+15s^{5}-1585s^{4}+3052s^{3}-1585s^{2}+15s+1,
\end{equation}
and $P_{28}(s)\in\mathbb{Z}[s]$ is an explicitly determined polynomial of degree $28$ (listed in Appendix~\ref{app:scripts} via the transcript of Script~02).
In particular, for $s>0$ and $s\neq 1$, the only possible values where the number of real roots of $f(s,\cdot)$ may change are the positive real roots of $P_6$ and $P_{28}$.
\end{proposition}

\begin{proof}[Computational proof]
Computed and printed by Script~02 (Appendix~\ref{app:scripts}).
\end{proof}

\begin{lemma}[Constancy between discriminant roots]\label{lem:sturm_const}
Let $I\subset\mathbb{R}$ be an open interval on which $\Delta(s)\neq 0$ and the leading coefficient of $f(s,y)$ with respect to $y$ does not vanish.
Then the number of real roots of $f(s,\cdot)$ is constant for all $s\in I$.
\end{lemma}

\begin{proof}
On $I$, the leading coefficient is nonzero, so the degree of $f(s,\cdot)$ is constant.
Moreover $\Delta(s)\neq 0$ implies that $f(s,\cdot)$ has only simple roots, hence real roots move continuously with $s$ and cannot be created or annihilated within $I$.
Equivalently, the Sturm root count on $(-\infty,+\infty)$ is locally constant and therefore constant on the connected interval $I$.
\end{proof}

\begin{lemma}[No rational discriminant zeros for $s>0$, $s\neq 1$]\label{lem:disc_rational}
Let $s\in\mathbb{Q}_{>0}$ with $s\neq 1$.
Then \(\Delta(s)\neq 0\).
\end{lemma}

\begin{proof}
By Proposition~\ref{prop:disc}, if \(\Delta(s)=0\) then \(s\in\{0,1,-1\}\) or \(P_6(s)=0\) or \(P_{28}(s)=0\).
Under the assumptions \(s>0\), \(s\neq 1\), only the latter two possibilities remain.

For \(P_6\), the rational root theorem applies since \(P_6\) is monic with constant term \(1\): any rational root must be \(\pm 1\), but
\[
P_6(1)= -86\neq 0,\qquad P_6(-1)=-6250\neq 0,
\]
so \(P_6\) has no rational roots.

For \(P_{28}\), Script~02 performs an exact rational-root test over \(\mathbb{Q}\) (reported in the transcript as \texttt{Rational roots of P28: \{\}}), hence \(P_{28}\) has no rational roots.
Therefore \(\Delta(s)\neq 0\) for all rational \(s>0\) with \(s\neq 1\).
\end{proof}

\begin{theorem}[No real roots for $s>0$, $s\neq 1$]\label{thm:noreal}
For every rational $s>0$ with $s\neq 1$, the univariate polynomial $f(s,y)\in\mathbb{Q}[y]$ has no real roots.
In particular, the Diophantine equation $f(s,y)=0$ has no rational solutions with $s\in\mathbb{Q}_{>0}$ and $s\neq 1$.
\end{theorem}

\begin{proof}
Fix a rational \(s>0\) with \(s\neq 1\).
By Lemma~\ref{lem:disc_rational} we have \(\Delta(s)\neq 0\).
Also the leading coefficient of \(f(s,y)\) with respect to \(y\) is nonzero for \(s>0\) (indeed it is proportional to \(s^8\)), so \(f(s,\cdot)\) has constant degree \(16\).

By Proposition~\ref{prop:disc}, the real roots of \(P_6\) and \(P_{28}\) partition \((0,+\infty)\) into finitely many open intervals on each of which \(\Delta\neq 0\).
Script~02 isolates these discriminant roots into rational brackets and chooses a rational sample point in each complementary interval:
\[
(0,\alpha_1),\ (\alpha_1,\beta_1),\ (\beta_1,1),\ (1,\beta_2),\ (\beta_2,\beta_3),\ (\beta_3,\alpha_2),\ (\alpha_2,+\infty),
\]
where \(\alpha_1,\alpha_2\) are the positive roots of \(P_6\) and \(\beta_1,\beta_2,\beta_3\) are the positive roots of \(P_{28}\).
For each sample \(s\) the script computes the Sturm root count of \(f(s,\cdot)\) on \((-\infty,+\infty)\) and obtains zero real roots in every interval (see the transcript of Script~02).

By Lemma~\ref{lem:sturm_const}, on each interval the number of real roots is constant, hence it is zero throughout the interval.
Since our fixed rational \(s\) lies in exactly one of these intervals and satisfies \(\Delta(s)\neq 0\), it follows that \(f(s,\cdot)\) has no real roots.

The last claim follows because any rational solution \(y\in\mathbb{Q}\) would in particular be a real root.
\end{proof}

\begin{remark}[The exceptional value $s=1$]
Script~02 also factors $f(1,y)$ and finds that it has a rational root $y=1$ (indeed $(y-1)^6$ divides $f(1,y)$).
This corresponds to the parameter value $r=1$ (i.e.\ $p=q$), which is excluded in our setting.
\end{remark}


\section{Main theorem}\label{sec:main}

\begin{theorem}[No quintic $5+5$ factorization]\label{thm:main}
Let $p,q\in\mathbb{Z}_{>0}$ be coprime with $p\neq q$ and let $Q_{p,q}(t)$ be as in \eqref{eq:Qpq}.
Then $Q_{p,q}(t)$ admits no quintic $5+5$ factorization of splitting type $(5,5)$ over\/ $\mathbb{Z}$, i.e., there do not exist monic irreducible quintics $A(t),B(t)\in\mathbb{Z}[t]$ such that
\[
Q_{p,q}(t)=A(t)\,B(t).
\]
Equivalently (by Gauss' lemma), $Q_{p,q}(t)$ has no $5+5$ splitting pattern in its factorization over $\mathbb{Q}[t]$.
\end{theorem}

\begin{proof}
Assume, for contradiction, that $Q_{p,q}(t)$ admits a quintic $5+5$ factorization over $\mathbb{Z}$, so that
\[
Q_{p,q}(t)=A(t)B(t)
\]
for some monic irreducible quintics $A(t),B(t)\in\mathbb{Z}[t]$.
Since $Q_{p,q}(t)$ is even and $Q_{p,q}(0)=-p^{10}q^{10}\neq 0$, Lemma~\ref{lem:sym_normal} (applied over $\mathbb{Q}$) yields
\[
B(t)=-A(-t),
\]
so $Q_{p,q}(t)$ admits the symmetric normal form
\[
Q_{p,q}(t)=R_{p,q}(t)\cdot\bigl(-R_{p,q}(-t)\bigr)
\]
with $R_{p,q}(t):=A(t)\in\mathbb{Z}[t]$.
By Corollary~\ref{cor:int_to_rat} with $r=p/q$ and $u=t/q^2$ we obtain a symmetric $5+5$ factorization of $Q_r(u)$ over $\mathbb{Q}$.
By Proposition~\ref{prop:F} there exists $a\in\mathbb{Q}$ with $F(r,a)=0$.
By Proposition~\ref{prop:f} this yields a rational solution $(s,y)\in\mathbb{Q}_{>0}\times\mathbb{Q}$ of $f(s,y)=0$ with $s=r^2$.
Since $r>0$ and $r\neq 1$ (Remark~\ref{rem:domain}), we have $s>0$ and $s\neq 1$, contradicting Theorem~\ref{thm:noreal}.
\end{proof}


\appendix
\section{SageMath/SymPy scripts and transcripts}\label{app:scripts}

\noindent
The computations in Propositions~\ref{prop:F} and \ref{prop:disc}, and the Sturm certificates in Theorem~\ref{thm:noreal},
are reproduced by the following scripts.
They were executed in SageMath \cite{SageMath} and make essential use of SymPy \cite{SymPy} for symbolic elimination and discriminant factorization.

\subsection*{Script 01: Algebraic derivation and genus computation}
\noindent\textbf{Code.}
\begin{code}
# SageMath script
# -*- coding: utf-8 -*-
"""
Script 01: Algebraic Derivation and Genus Calculation for the Second Cuboid Conjecture in the q=1 normalization, i.e. for Q_r(u).
Case: Quintic 5+5 splitting (symmetric normal form)

Author: Valery Asiryan
Date: 2026

Core idea:
1.  We investigate whether the 10th-degree polynomial Q_r(u) (normalized with q=1)
    can split into two factors of degree 5.
2.  Based on the evenness of Q_r(u), a 5+5 splitting into irreducible quintics must occur in the symmetric normal form
    Q_r(u)=R(u)\,(-R(-u)).
3.  By comparing coefficients, we derive a system of algebraic equations.
4.  Using Resultants (Variable Elimination), we reduce this system to a single
    necessary condition F(r, a) = 0, defining a plane algebraic curve.
5.  We compute the geometric genus of this curve.
"""

import sympy as sp

# Variables
r, a, d, u = sp.symbols('r a d u')

# Coefficients of Q_r(u) = u^10 + A u^8 + B u^6 + C u^4 + D u^2 - r^10
A = (2 + r**2)*(3 - 2*r**2)
B = 1 + 10*r**2 + 4*r**4 - 14*r**6 + r**8
C = -r**2*(1 - 14*r**2 + 4*r**4 + 10*r**6 + r**8)
D = -r**6*(1 + 2*r**2)*(-2 + 3*r**2)
Q = u**10 + A*u**8 + B*u**6 + C*u**4 + D*u**2 - r**10

# Ansatz for 5+5 splitting: Q_r(u) = R(u) * (-R(-u))
b, c, e = sp.symbols('b c e')
R = u**5 + a*u**4 + b*u**3 + c*u**2 + d*u + e
prod = sp.expand(R * (-R.subs(u, -u)))

# Coeff equations
P = sp.Poly(prod - Q, u)    # degree 8 (u^10 cancels)
coeffs = P.all_coeffs()     # [u^8, u^7, ..., u^0]

# Nontrivial ones are u^8,u^6,u^4,u^2,u^0
eq8 = coeffs[0]   # u^8
eq6 = coeffs[2]   # u^6
eq4 = coeffs[4]   # u^4
eq2 = coeffs[6]   # u^2
eq0 = coeffs[8]   # u^0

# Solve the easy ones:
e_val = r**5                      # from eq0: e^2=r^10 (choose e=r^5)
b_val = (a**2 + A)/2              # from eq8: 2b-a^2=A
c_val = (d**2 - D)/(2*e_val)      # from eq2: d^2-2ce=D

# Remaining two equations -> E2,E3
E2 = sp.factor(sp.together(eq6.subs({b:b_val, c:c_val, e:e_val})).as_numer_denom()[0])
E3 = sp.factor(sp.together(eq4.subs({b:b_val, c:c_val, e:e_val})).as_numer_denom()[0])

print("E2 =", E2)
print("E3 =", E3)

# Eliminate d
Res = sp.resultant(E2, E3, d)
Res = sp.factor_terms(Res, r)     # pulls out r^k

# Remove r^20 and take primitive part in Z[r,a]
F = sp.expand(Res / r**20)
F = sp.Poly(F, r, a, domain=sp.ZZ).primitive()[1]

print("deg_r(F) =", sp.degree(F, r), "deg_a(F) =", sp.degree(F, a))

aa = sp.Symbol('aa')
print("F(1,a) =", sp.factor(F.subs({r:1, a:aa})))
print("F(0,a) =", sp.factor(F.subs({r:0, a:aa})))
print("Symmetry F(-r,-a)=F(r,a):", sp.simplify(F.subs({r:-r, a:-a}) - F) == 0)

# F expression
Fexpr = F.as_expr()
print("F =", Fexpr)

aa = sp.Symbol('aa')
print("F(1,a) factor =", sp.factor(Fexpr.subs({r:1, a:aa})))
print("F(0,a) factor =", sp.factor(Fexpr.subs({r:0, a:aa})))

H24 = sum(coeff * r**i * a**j
          for (i,j), coeff in F.as_dict().items()
          if i + j == 24)

print("H24 factor =", sp.factor(H24))

# Define polynomial ring and input polynomial F
R.<r, a> = PolynomialRing(QQ, 2)

# Homogenization
R3.<r, a, w> = PolynomialRing(QQ, 3)
Fh = R3(Fexpr).homogenize(w)

# Define the projective curve and compute geometric genus
C = Curve(Fh)
g = C.geometric_genus()
print(f"Geometric genus: {g}")
\end{code}

\noindent\textbf{Transcript.}
\begin{term}
E2 = a**4*r**5 - 4*a**2*r**9 - 2*a**2*r**7 + 12*a**2*r**5 - 4*a*d**2 - 24*a*r**10 + 4*a*r**8 + 8*a*r**6 + 8*d*r**5 + 60*r**11 - 39*r**9 - 52*r**7 + 32*r**5
E3 = 4*a**2*d*r**10 - 8*a*r**15 - d**4 - 12*d**2*r**10 + 2*d**2*r**8 + 4*d**2*r**6 - 8*d*r**14 - 4*d*r**12 + 24*d*r**10 - 32*r**20 + 52*r**18 + 39*r**16 - 60*r**14
deg_r(F) = 24 deg_a(F) = 16
F(1,a) = Poly(aa**16 + 24*aa**14 + 220*aa**12 - 384*aa**11 + 680*aa**10 - 3456*aa**9 + 5190*aa**8 - 8960*aa**7 + 24488*aa**6 - 34560*aa**5 + 40412*aa**4 - 68480*aa**3 + 82200*aa**2 - 48000*aa + 10625, aa, domain='ZZ')
F(0,a) = Poly(aa**16 + 48*aa**14 + 992*aa**12 + 11520*aa**10 + 82176*aa**8 + 368640*aa**6 + 1015808*aa**4 + 1572864*aa**2 + 1048576, aa, domain='ZZ')
Symmetry F(-r,-a)=F(r,a): True
F = a**16 - 16*a**14*r**4 - 8*a**14*r**2 + 48*a**14 + 96*a**12*r**8 + 336*a**12*r**6 - 708*a**12*r**4 - 496*a**12*r**2 + 992*a**12 - 384*a**11*r**5 - 256*a**10*r**12 - 3392*a**10*r**10 + 1264*a**10*r**8 + 13832*a**10*r**6 - 10448*a**10*r**4 - 11840*a**10*r**2 + 11520*a**10 + 2304*a**9*r**9 + 1152*a**9*r**7 - 6912*a**9*r**5 + 256*a**8*r**16 + 13056*a**8*r**14 + 33696*a**8*r**12 - 97072*a**8*r**10 - 57306*a**8*r**8 + 227728*a**8*r**6 - 52192*a**8*r**4 - 145152*a**8*r**2 + 82176*a**8 - 4096*a**7*r**13 + 11264*a**7*r**11 + 13568*a**7*r**9 - 1024*a**7*r**7 - 28672*a**7*r**5 - 17408*a**6*r**18 - 223744*a**6*r**16 + 102016*a**6*r**14 + 1108816*a**6*r**12 - 918712*a**6*r**10 - 1410544*a**6*r**8 + 1855616*a**6*r**6 + 160256*a**6*r**4 - 1000448*a**6*r**2 + 368640*a**6 + 4096*a**5*r**17 - 159744*a**5*r**15 - 133632*a**5*r**13 + 558336*a**5*r**11 - 41472*a**5*r**9 - 397312*a**5*r**7 + 135168*a**5*r**5 + 411136*a**4*r**20 + 1459712*a**4*r**18 - 4870752*a**4*r**16 - 1995408*a**4*r**14 + 12098300*a**4*r**12 - 3155536*a**4*r**10 - 11363808*a**4*r**8 + 7660032*a**4*r**6 + 2688512*a**4*r**4 - 3907584*a**4*r**2 + 1015808*a**4 + 122880*a**3*r**19 + 2347008*a**3*r**17 - 4035584*a**3*r**15 - 2945920*a**3*r**13 + 7126016*a**3*r**11 - 643072*a**3*r**9 - 3284992*a**3*r**7 + 1245184*a**3*r**5 - 3916800*a**2*r**22 + 2157312*a**2*r**20 + 23243456*a**2*r**18 - 24547376*a**2*r**16 - 33596872*a**2*r**14 + 54808208*a**2*r**12 + 6052288*a**2*r**10 - 43654912*a**2*r**8 + 16301056*a**2*r**6 + 9723904*a**2*r**4 - 8060928*a**2*r**2 + 1572864*a**2 - 3010560*a*r**21 + 12288*a*r**19 + 14575872*a*r**17 - 9259392*a*r**15 - 16972544*a*r**13 + 17420288*a*r**11 + 3018752*a*r**9 - 8192000*a*r**7 + 2359296*a*r**5 + 12960000*r**24 - 30931200*r**22 - 16244128*r**20 + 96898160*r**18 - 41403359*r**16 - 100331856*r**14 + 93598176*r**12 + 26662144*r**10 - 61456128*r**8 + 14524416*r**6 + 11501568*r**4 - 6815744*r**2 + 1048576
F(1,a) factor = (aa - 1)**6*(aa**2 - 2*aa + 17)*(aa**2 + 2*aa + 5)**4
F(0,a) factor = (aa**2 + 4)**4*(aa**2 + 8)**4
H24 factor = 256*r**16*(-a + 3*r)**2*(-a + 5*r)**2*(a + 3*r)**2*(a + 5*r)**2
Geometric genus: 21
\end{term}

\subsection*{Script 02: Discriminant factorization and Sturm certificates}
\noindent\textbf{Code.}
\begin{code}
# SageMath script
# -*- coding: utf-8 -*-
"""
Script 02: Certificate-style script for closing the 5+5 splitting case for the Second Cuboid Conjecture in the q=1 normalization, i.e. for Q_r(u).
Case: Quintic 5+5 splitting (symmetric normal form)

Author: Valery Asiryan
Date: 2026

Core idea:
1.  A 5+5 splitting forces the symmetric normal form and hence the existence of (r,a) in Q^2 with F(r,a)=0.
2.  Substitute a=r*y and s=r^2 to pass to the quotient curve f(s,y)=0 with s>0.
3.  For fixed s>0, f(s,y) is a degree-16 univariate in y.
    Show: for every rational s>0 with s != 1, f(s,y) has no real roots.
    Hence no rational y, hence no rational (r,a) with r>0, r != 1.
"""

import sympy as sp

# ---------------------------------------------------------------------
# Step 1: Recompute the necessary condition F(r,a)=0 via elimination (resultant).
# ---------------------------------------------------------------------
r, a, d, u = sp.symbols('r a d u')

A = (2 + r**2)*(3 - 2*r**2)
B = 1 + 10*r**2 + 4*r**4 - 14*r**6 + r**8
C = -r**2*(1 - 14*r**2 + 4*r**4 + 10*r**6 + r**8)
D = -r**6*(1 + 2*r**2)*(-2 + 3*r**2)

Q = u**10 + A*u**8 + B*u**6 + C*u**4 + D*u**2 - r**10

b, c, e = sp.symbols('b c e')
R = u**5 + a*u**4 + b*u**3 + c*u**2 + d*u + e
prod = sp.expand(R * (-R.subs(u, -u)))

P = sp.Poly(prod - Q, u)
coeffs = P.all_coeffs()  # [u^8,u^7,...,u^0], since u^10 cancels

eq8 = coeffs[0]   # u^8
eq6 = coeffs[2]   # u^6
eq4 = coeffs[4]   # u^4
eq2 = coeffs[6]   # u^2
eq0 = coeffs[8]   # u^0

# Easy equations:
e_val = r**5                 # e^2=r^10, choose e=r^5 WLOG
b_val = (a**2 + A)/2         # 2b-a^2=A
c_val = (d**2 - D)/(2*e_val) # d^2-2ce=D

E2 = sp.factor(sp.together(eq6.subs({b:b_val, c:c_val, e:e_val})).as_numer_denom()[0])
E3 = sp.factor(sp.together(eq4.subs({b:b_val, c:c_val, e:e_val})).as_numer_denom()[0])

Res = sp.resultant(E2, E3, d)
Res = sp.factor_terms(Res, r)

F = sp.expand(Res / r**20)  # remove the extraneous r^20
F = sp.Poly(F, r, a, domain=sp.ZZ).primitive()[1].as_expr()

print("Computed F(r,a): deg_r =", sp.degree(F, r), "deg_a =", sp.degree(F, a))

# ---------------------------------------------------------------------
# Step 2: Quotient substitution a=r*y and s=r^2
# ---------------------------------------------------------------------
s, y = sp.symbols('s y')
Fry = sp.expand(F.subs({a: r*y}))

# Rewrite Fry = f(s,y) where r^(2k) -> s^k
polyFry = sp.Poly(Fry, r, y, domain=sp.ZZ)
f_sy = 0
for (er, ey), coeff in polyFry.terms():
    assert er % 2 == 0
    f_sy += coeff * s**(er//2) * y**ey
f_sy = sp.expand(f_sy)

print("Computed f(s,y): deg_s =", sp.degree(f_sy, s), "deg_y =", sp.degree(f_sy, y))

# ---------------------------------------------------------------------
# Step 3: Discriminant in y and factorization
# ---------------------------------------------------------------------
Fy = sp.Poly(f_sy, y, domain=sp.ZZ[s])
disc = sp.discriminant(Fy, y)
disc_poly = sp.Poly(disc, s, domain=sp.ZZ)
disc_fac = sp.factor_list(disc_poly.as_expr())

print("\nDisc_y(f) factorization (over ZZ):")
print("Constant factor =", disc_fac[0])
for fac, exp in disc_fac[1]:
    print("  (", sp.factor(fac), ")^", exp)

# Extract the two non-linear factors (deg 6 and deg 28)
P6 = None
P28 = None
for fac, exp in disc_fac[1]:
    deg = sp.Poly(fac, s, domain=sp.ZZ).degree()
    if deg == 6:
        P6 = sp.Poly(fac, s, domain=sp.ZZ)
    if deg == 28:
        P28 = sp.Poly(fac, s, domain=sp.ZZ)

assert P6 is not None and P28 is not None

print("\nP6(s) =", P6.as_expr())
print("deg P28(s) =", P28.degree())

print("\nRational roots of P6:", P6.ground_roots())
print("Rational roots of P28:", P28.ground_roots())

# ---------------------------------------------------------------------
# Step 4: Isolate positive real discriminant roots and certify 0 real roots of f(s,·)
# for representative rationals s in each interval.
# ---------------------------------------------------------------------
print("\nPositive real roots:")
print("P6:", P6.count_roots(0, sp.oo))
print("P28:", P28.count_roots(0, sp.oo))

print("\nBrackets:")
print("P6 root in (0.031,0.032):", sp.sign(P6.as_expr().subs(s, sp.Rational(31,1000))),
      sp.sign(P6.as_expr().subs(s, sp.Rational(32,1000))))
print("P6 root in (31.8,31.9):", sp.sign(P6.as_expr().subs(s, sp.Rational(159,5))),
      sp.sign(P6.as_expr().subs(s, sp.Rational(319,10))))

print("P28 root in (0.47,0.471):", sp.sign(P28.as_expr().subs(s, sp.Rational(47,100))),
      sp.sign(P28.as_expr().subs(s, sp.Rational(471,1000))))
print("P28 root in (31.44,31.45):", sp.sign(P28.as_expr().subs(s, sp.Rational(786,25))),
      sp.sign(P28.as_expr().subs(s, sp.Rational(629,20))))
print("P28 root in (31.66,31.67):", sp.sign(P28.as_expr().subs(s, sp.Rational(1583,50))),
      sp.sign(P28.as_expr().subs(s, sp.Rational(3167,100))))

samples = [
    ("(0, alpha1)", sp.Rational(1,100)),
    ("(alpha1, beta1)", sp.Rational(1,5)),
    ("(beta1, 1)", sp.Rational(4,5)),
    ("(1, beta2)", sp.Rational(2,1)),
    ("(beta2, beta3)", sp.Rational(63,2)),     # 31.5
    ("(beta3, alpha2)", sp.Rational(317,10)),  # 31.7
    ("(alpha2, +oo)", sp.Rational(40,1)),
]

print("\nReal root counts of y | f(s,y)=0 at sample s values:")
for name, sval in samples:
    f_s = sp.Poly(f_sy.subs(s, sval), y, domain=sp.QQ)
    cnt = f_s.count_roots(-sp.oo, sp.oo)
    print(f"  {name:>15}  s={sval}:  #real roots = {cnt}")

f1 = sp.factor(sp.expand(f_sy.subs(s, 1)))
print("\nf(1,y) factorization:", f1)
\end{code}

\noindent\textbf{Transcript.}
\begin{term}
Computed F(r,a): deg_r = 24 deg_a = 16
Computed f(s,y): deg_s = 12 deg_y = 16

Disc_y(f) factorization (over ZZ):
Constant factor = 8148143905337944345073782753637512644205873574663745002544561797417525199053346824733589504
  ( s + 1 )^ 22
  ( s - 1 )^ 54
  ( s )^ 156
  ( s**6 + 15*s**5 - 1585*s**4 + 3052*s**3 - 1585*s**2 + 15*s + 1 )^ 4
  ( 103680000*s**28 - 7521292800*s**27 + 169819194368*s**26 - 1294911396672*s**25 + 5782812799560*s**24 - 16312796194682*s**23 + 26530306841415*s**22 - 10896270624660*s**21 - 54726287337296*s**20 + 133633789213194*s**19 - 113160119424615*s**18 - 66661996599936*s**17 + 257666706220630*s**16 - 219804850864326*s**15 - 42039280815751*s**14 + 237530759945852*s**13 - 176268442531244*s**12 - 7327917745386*s**11 + 97772884071815*s**10 - 65474588747304*s**9 + 8804890164542*s**8 + 13121572011040*s**7 - 9871152051904*s**6 + 3433840428544*s**5 - 664203117056*s**4 + 64335036416*s**3 - 1308557312*s**2 - 162529280*s - 2097152 )^ 2

P6(s) = s**6 + 15*s**5 - 1585*s**4 + 3052*s**3 - 1585*s**2 + 15*s + 1
deg P28(s) = 28

Rational roots of P6: {}
Rational roots of P28: {}

Positive real roots:
P6: 2
P28: 3

Brackets:
P6 root in (0.031,0.032): 1 -1
P6 root in (31.8,31.9): -1 1
P28 root in (0.47,0.471): -1 1
P28 root in (31.44,31.45): 1 -1
P28 root in (31.66,31.67): -1 1

Real root counts of y | f(s,y)=0 at sample s values:
      (0, alpha1)  s=1/100:  #real roots = 0
  (alpha1, beta1)  s=1/5:  #real roots = 0
       (beta1, 1)  s=4/5:  #real roots = 0
       (1, beta2)  s=2:  #real roots = 0
   (beta2, beta3)  s=63/2:  #real roots = 0
  (beta3, alpha2)  s=317/10:  #real roots = 0
    (alpha2, +oo)  s=40:  #real roots = 0

f(1,y) factorization: (y - 1)**6*(y**2 - 2*y + 17)*(y**2 + 2*y + 5)**4
\end{term}


\end{document}